\documentclass[11pt]{article}

\usepackage{amsfonts}
\usepackage{amssymb}
\usepackage{url,graphicx,tabularx,array,geometry}
\usepackage{amsmath}
\usepackage{bbm} 
\usepackage{color}
\usepackage{epstopdf}
\usepackage{framed}
\usepackage{verbatim}
\usepackage{blkarray} 
\usepackage{enumitem} 
\usepackage{multirow} 
\usepackage{blkarray}
\usepackage{kbordermatrix}
\usepackage[numbers]{natbib}
\usepackage{float}
\usepackage{tikz}
\usepackage{hyperref}

\usepackage{amsthm} 

\usepackage[english]{babel}
\usepackage[utf8]{inputenc}
\usepackage{algorithm}
\usepackage[noend]{algpseudocode}

\newtheorem{theorem}{Theorem}[section]
\newtheorem{lemma}[theorem]{Lemma}
\newtheorem{proposition}[theorem]{Proposition}
\newtheorem{corollary}[theorem]{Corollary}

\newtheorem{example}[theorem]{Example}
\theoremstyle{remark}

\newcommand{\R}{\mathbb{R}}

\begin{document}

\title{Special cases of the quadratic shortest path problem}
\author{Hao Hu \thanks{CentER, Department of Econometrics and OR, Tilburg University, The Netherlands, {\tt h.hu@uvt.nl}}
\and {R. Sotirov \thanks{Department of Econometrics and OR, Tilburg University, The Netherlands, {\tt r.sotirov@uvt.nl}}} }
\date{}

\maketitle
\begin{abstract}
The  quadratic shortest path problem (QSPP) is \textcolor{black}{the problem of finding a path with prespecified start
vertex $s$ and end vertex $t$ in a digraph} such that
the sum of weights of arcs and the sum of interaction costs over all pairs of arcs on the path is minimized.
We first consider a variant  of the QSPP known as the adjacent QSPP.
It was recently proven that the adjacent QSPP on cyclic digraphs cannot be approximated unless P=NP.
Here, we give a simple proof for the same result.

We also show that if the quadratic cost matrix is a symmetric weak sum matrix \textcolor{black}{ and all $s$-$t$ paths have the same length,}
then an optimal solution for the QSPP  can be obtained by solving the corresponding  instance of the shortest path problem.
Similarly, it is shown that the QSPP with a symmetric product cost matrix is solvable in polynomial time.

Further, we provide sufficient and necessary conditions for a QSPP instance on a complete symmetric digraph with four vertices to be linearizable.
We also characterize linearizable QSPP instances on complete symmetric digraphs with more than four vertices.
Finally, we derive an algorithm that examines whether  a QSPP instance on the directed grid graph $G_{pq}$ ($p,q\geq 2$) is linearizable.
The complexity of this  algorithm  is ${\mathcal{O}(p^{3}q^{2}+p^{2}q^{3})}$.
\end{abstract}


\section{Introduction}
The shortest path problem (SPP) is the problem of finding a path between two vertices in a directed graph such that the total weight of the arcs on the path is minimized.
The quadratic  shortest path problem (QSPP) is the problem of finding a  path between two  vertices  in a directed graph such that   the total weight of the arcs and
the sum of interaction costs over all pairs of arcs on the path is minimized.

The SPP is a well-studied combinatorial optimization problem, that  can be solved in polynomial time if there do not exist negative cycles in the considered graph.
There exist several efficient algorithms for solving the shortest path problem, e.g., the  Dijkstra algorithm \cite{dijkstra1959note} and
the Floyd--Warshall algorithm \cite{Floyd,Warshall}.
The SPP can be applied to various problems such as transportation planning, network protocols, plant and facility layout, robotics, VLSI design etc.
On the other hand,   there are not many results on the quadratic shortest path problem.
In the recent paper by Rostami et al.~\cite{rostami2016quadratic}
it is proven that  the general QSPP cannot be approximated unless P=NP.
The same result is proven for  the adjacent QSPP (AQSPP), that is a variant of the QSPP.
In the AQSPP  interaction  costs of all non-adjacent pairs of arcs are equal to zero, see \cite{rostami2016quadratic}.
However, the adjacent QSPP is solvable in polynomial time for acyclic graphs and \textcolor{black}{series-parallel graphs, see \cite{rostami2016quadratic}.}

Although the QSPP was only recently introduced,
several variants of the SPP that are related to the QSPP were studied in  \cite{sivakumar1994variance,Sen:01}.
In particular  Sivakumar and  Batta \cite{sivakumar1994variance} consider a variance-constrained shortest path, and  Sen et al.~\cite{Sen:01} a route-planning model
in which the choice of a route is based on the mean as well as the variance of the path travel-time.
The QSPP is also related to the reliable shortest path problem, see e.g., Nie and Wu \cite{NieWu}.
The QSPP appears in a study on network protocols.
Namely, Murakami and Kim \cite{murakami1997comparative} study different restoration schemes of survivable asynchronous transfer mode  networks that can be formulated as the QSPP.
For \textcolor{black}{a detailed overview} on applications of the QSPP see \cite{rostami2016quadratic}.

Buchheim and Traversi \cite{buchheim2015quadratic}, and  Rostami et al.~\cite{rostami2016quadratic} present several approaches  to solve the general QSPP.
In particular, the authors  in \cite{buchheim2015quadratic} consider separable underestimators that can be exploited for solving binary quadratic programming problems, including the QSPP.
Several lower bounding approaches for  the QSPP, including a Glimore-Lawler-type  bound and  reformulation-based bound are presented in \cite{rostami2016quadratic}.
In this paper we do not investigate computational aspects for solving the QSPP in general. \\

\noindent
\textbf{Main results and outline.}\\
In Section \ref{sect:probl_form}, we formulate the quadratic shortest path problem as an integer programming problem.
Complexity results for the general and adjacent QSPP are given in Section \ref{section:NPhard}.
In particular, in Section \ref{section:general} we derive a new polynomial-time reduction from the well known quadratic assignment problem (QAP)  to the QSPP.
Our reduction differs from the one given by Rostami et al.~\cite{frey2015quadratic}. Namely, our approach results in an instance for the QSPP with $n^2$ arcs,
while the reduction from  \cite{frey2015quadratic} derives an instance with $\mathcal{O}(n^3)$ arcs. Here, $n$ is the order of the data matrices in the  quadratic assignment problem.
The here presented polynomial-time reduction from the QAP, in combination with the library of the QAP \cite{QAPlib},
provides a source of difficult QSPP test instances.

In Section \ref{section:AQSPP}, we describe the polynomial-time algorithm for solving the adjacent QSPP from  \cite{frey2015quadratic}.
We also show that the algorithm  fails for the adjacent QSPP  considered on directed cyclic graphs, while it performs well on directed acyclic graphs (DAGs).
Further, we  provide a polynomial-time  reduction from the 2-arc-disjoint paths problem, that is known to be NP-complete, to the adjacent QSPP considered on a directed cyclic graph.
Our proof of inapproximability is \textcolor{black}{considerably simpler}  than the proof from  \cite{rostami2016quadratic}.

In Section \ref{section:linearQSPP}, we consider  special cases of the QSPP.
{\color{black}We first consider linearizable instances; that is,  when there exists a corresponding instance of the SPP such that
the associated costs for both problems are equal for every feasible path.
It is easy to see that the QSPP considered on a directed cycle is linearizable.}
Here, we also show that a QSPP instance on a  digraph for which every $s$-$t$ path has the same length and
 whose quadratic cost matrix is a symmetric weak sum matrix is linearizable.
Finally, we prove that a solution of the QSPP whose  quadratic cost matrix together with linear costs included on the diagonal
is a nonnegative symmetric product matrix can be obtained by solving the corresponding SPP.

We provide sufficient and necessary conditions for an instance of the  QSPP on a complete digraph with four nodes to be linearizable in Section \ref{sect:completeGraphs}.
In the same section we give  several properties of linearizable QSPP instance on complete digraphs with more than four nodes.

In Section \ref{sect:DGG}, we present an algorithm that examines whether a QSPP instance on the directed grid graph $G_{p,q}$ ($p,q\geq 2$) is linearizable.
If the instance is linearizable, then our algorithm  provides the corresponding linearization vector.
The complexity of the algorithm is   ${\mathcal{O}(p^{3}q^{2}+p^{2}q^{3})}$.

\section{Problem formulation} \label{sect:probl_form}

Let $G = (V,A)$ be a directed graph with vertex set $V$ ($|V|=n$) and arc set $A$ ($|A|=m$).
A walk is defined as an ordered set of  vertices  ${(v_{1},\ldots,v_{k})}$, $k> 1$ such that $(v_{i},v_{i+1})\in A$ for $i=1,\ldots,k-1$.
The length of a walk equals to the number of visited arcs. A walk is called a path if it does not contain repeated  vertices.
Given a source vertex ${s\in V}$ and a target vertex $t\in V$, a  $s$-$t$ path is a path ${P =(v_{1},v_{2},\ldots,v_{k})}$
such that $v_{1}= s$ and $v_{k} = t$.

The  quadratic  cost of a $s$-$t$ path $P$ is calculated as follows.
We are given a nonnegative vector $c \in \R_{+}^{m}$ indexed by the arc set $A$,
and a nonnegative  symmetric matrix of order $m$  $Q=(q_{e,f})$ with zero-diagonal
whose rows and columns are   indexed by  the arc set.
An arc $e \in P$ has the linear cost  (weight) $c_{e}$, and a pair of arcs $e,f \in P$
the interaction  cost $2 q_{ef}$. The total cost of  the $s$-$t$ path $P$ is given by
\begin{equation}\label{costequation}
C(P,c,Q) = \sum_{e,f \in P} q_{ef}  + \sum_{e\in P} c_{e}.
\end{equation}
If $Q$ is a zero-matrix, then the cost of a $s$-$t$ path $P$ is denoted by $C(P,c)$.
We assume that the graph $G$ does not contain a directed cycle {\color{black} whose total cost is zero}.

Let us introduce the quadratic shortest path problem in a formal way.
Let $P$ be a $s$-$t$ path, and $x$ a binary vector of length $m$ such that $x_{ij}$ is one if the arc $(i,j)\in A$ is on the $s$-$t$ path $P$ and zero  otherwise.
Now, the quadratic  cost of the $s$-$t$ path $P$ with the characteristic vector $x$, is given by
\[
\sum_{(i,j),(k,l) \in A} q_{ij,kl} x_{ij}x_{kl} + \sum_{(i,j)\in A} c_{ij} x_{ij} = x^{\mathrm T}Qx + c^{\mathrm T}x.
\]
Given a vertex $i \in V$, the set of predecessor and successor  vertices  of
$i$ are denoted by  $\delta^{-}(i) := \{ j \in V \;| \; (j,i) \in A \}$
and $\delta^{+}(i) := \{ j \in V \;| \; (i,j) \in A  \}$, respectively.
The path polyhedron is defined as follows:
\begin{equation}\label{stPathPolytope}
P_{st}(G) := \{ x\in \R^{m} \; | \; \sum_{j \in \delta^{+}(i)} x_{ij} - \sum_{j \in \delta^{-}(i)} x_{ji} = b_{i} \;\; \forall i\in V,\;\;  0 \leq x \leq 1 \},
\end{equation}
where $b$ is a vector of length $n$ such that $b_{i} = 1$ if $i = s$, $b_{i} = -1$ if $i = t$,
and $b_{i} = 0$ if $i \in V \backslash \{s,t\}$.
Now  the QSPP can be modeled as the following quadratic integer programming problem:
\begin{equation}\label{orignalQAP}
\begin{aligned}
& \text{minimize}
& &  x^{\mathrm T}Qx + c^{\mathrm T}x\\
& \text{subject to}
& &  x \in P_{st}(G) \cap \{0,1\}^{m}. & \hspace{0.5cm} & \\
\end{aligned}
\end{equation}
\textcolor{black}{Since there are no cycles of cost zero in $G$, the optimal solution of \eqref{orignalQAP} is guaranteed to be an $s$-$t$ path.}
If $Q$ is a zero-matrix,  then the problem reduces to the shortest path problem.
Due to the flow conservation law, one can remove one of the equations in $P_{st}(G)$.

A  QSPP instance $\mathcal{I}$ (resp.~SPP instance  $\mathcal{I'}$)  can be specified by the tuple $\mathcal{I} = (G,s,t,c,Q)$ (resp.~$\mathcal{I'} = (G,s,t,c)$).
Note that we use  both, $e$ and $(i,j)$,  to denote an arc $e=(i,j)$.
Sometimes one is more convenient than another.

\section{Complexity results for the general and adjacent QSPP}\label{section:NPhard}

In Section \ref{section:general} we  present a polynomial-time reduction from the quadratic assignment problem  to the QSPP.
Our reduction results in a significantly smaller number of arcs in the constructed  QSPP instance,
than  the number of the arcs in the QSPP instance  provided by the  reduction from  \cite{frey2015quadratic}.
Section \ref{section:AQSPP} and \ref{section:AQSPP_disjoint} consider the adjacent  QSPP.
In Section \ref{section:AQSPP} we review the algorithm from  \cite{frey2015quadratic} on solving the AQSPP on directed graphs,
and show that it fails when the digraph under consideration contains a cycle.
Further, we provide a proof  showing that the AQSPP on cyclic digraphs cannot be approximated unless P=NP, see Section \ref{section:AQSPP_disjoint}.
Our proof is simpler than the proof from \cite{rostami2016quadratic}.

\subsection{The general QSPP} \label{section:general}

In \cite{frey2015quadratic} it is proven that the QSPP is NP-hard  by providing  a polynomial-time reduction from the QAP.
The size of so constructed QSPP instance in \cite{frey2015quadratic}   is considerably larger than the size of the input QAP instance.
In particular, if a QAP instance consists of $n$ facilities and $n$  locations, then the constructed  QSPP instance as described in
 \cite{frey2015quadratic} has $n^2+2$  vertices  and $n^3-2n^2+3n$ arcs.

We present here another polynomial-time reduction from the QAP to the QSPP.
Our reduction yields a QSPP instance with $n+1$  vertices  and $n^2$ arcs, where $n$ is the order of the data matrices in the QAP.
This enables us to derive  QSPP  test instances of reasonable sizes, from the QAP instances given in the QAP library \cite{QAPlib}.
Note that the QAP library contains solutions and/or bounds for many QAP instances,
 and is therefore a source of test instances for the QSPP.

The quadratic assignment problem is the following optimization problem:
\[
\min \left \{ \sum_{i,j,k,l}a_{ik} b_{jl}x_{ij}x_{kl} + \sum_{i,j}c_{ij}x_{ij}: ~X=(x_{ij}), ~X\in \Pi_n \right \},
\]
where $A=(a_{ik})$, $B=(b_{jl})$ are given symmetric $n\times n$ matrices,  $C=(c_{ij})\in \R^{n \times n}$,
and $\Pi_n$ is the set of $n\times n $ permutation matrices.

The quadratic assignment problem has the following interpretation.
Suppose that there are $n$ facilities and $n$ locations.
The flow between each pair of facilities, say $i,k$, and the distance between each pair of locations, say $j,l$,
are given by $a_{ik}$ and $b_{jl}$, respectively. The cost of placing a facility $i$ to location $j$ is $c_{ij}$.
The QAP problem is to find an assignment of facilities to locations such that the sum of the distances multiplied
by the corresponding flows together with the total  cost is minimized.

Let us now allow a directed multigraph in the definition of the QSPP problem.
A multigraph is a graph which is permitted to have multiple arcs \textcolor{black}{between any pair of vertices}.
We can now prove the following theorem.
\begin{theorem} \label{QAPtoQSPP2}
\textcolor{black}{There exists a polynomial time reduction from QAP to QSPP, i.e., $QAP \propto QSPP$.}
\end{theorem}
\begin{proof}
Suppose we are given a QAP instance with $n \times n$ input matrices $A,B,C$.
\textcolor{black}{Assume without loss of generality that the diagonal entries of $A,B$ are all equal to zero.
(Note that  one could ``shift'' the diagonal elements to the linear cost matrix $C$).}
We construct a QSPP instance on the graph $G = (V,A)$ whose vertex and arc sets are defined as follows:
\begin{align*}
V := & ~\{w_{j} : ~j = 1,\ldots,n+1\},\\
A := &  ~\{ (w_{j},w_{j+1})^{i}  :~  i,j =1,\ldots, n \},
\end{align*}
where the superscript $i$ indicates the $i$th arc between vertices $w_{j}$ and $w_{j+1}$.
The starting and ending vertices are $s = w_{1}$ and $t = w_{n+1}$, respectively.
	
The linear cost of the arc $e$ is defined as
	$$c_{e}  := \begin{cases}
	c_{ij} & \text{if } e = (w_{j},w_{j+1})^{i} \\
	0 & \text{otherwise}.
	\end{cases}$$
The interaction cost between the pair of arcs  $e = (w_{j},w_{j+1})^{i}$ and $f = (w_{l},w_{l+1})^{k}$ is defined:
	$$ q_{e,f} := \begin{cases}
	a_{ik}  \cdot b_{jl}  & \text{if } j \neq l \\
	M & \text{if } j=l ~{\rm and}~ i\neq k,
	\end{cases}$$
where $M$ is a big number. All the other pairs of arcs have zero interaction costs.

If we have a feasible QAP instance, say facility   $i$ is mapped into location $\pi(i)$,
then we take the feasible QSPP instance with arcs $(w_{\pi(i)},w_{\pi(i)+1})^{i}$  ($i = 1,\ldots,n$).
By construction, those two instances have the same objective value.
Conversely, every $s$-$t$ path in $G$ \textcolor{black}{with objective value less than $M$} corresponds to the assignment for  the  QAP  with the same objective value.
\end{proof}
It follows from the above construction that for a given input instance of the QAP with $n\times n$ data matrices,
one can construct a  QSPP instance with $n+1$  vertices  and $n^2$ arcs.
\textcolor{black}{We note that the polynomial time reduction from Theorem \ref{QAPtoQSPP2}  is also valid for the undirected version of the QSPP.}

\subsection{The adjacent QSPP  restricted to DAGs}\label{section:AQSPP}

The adjacent QSPP  is a variant of the QSPP, where interaction costs of all non-adjacent pairs of arcs are equal to zero.
In other words, only the interaction  cost of the form $q_{ij,kl}$ with $j = k$ and $i\neq l$, or with $i = l$  and $j \neq k$ can have nonzero value.
A polynomial time algorithm that solves instances of the adjacent  QSPP on directed acyclic graphs is presented in  \cite{frey2015quadratic}.
It is actually stated  in  \cite{frey2015quadratic}  that the  proposed algorithm finds an optimal solution for the AQSPP on any digraph in polynomial time,
which turns to be true only for  directed acyclic graphs.

In this section we first describe  the  approach from \cite{frey2015quadratic}, and then provide an example to show that the algorithm
fails if the graph under consideration is not acyclic.

Let $G$ be a directed acyclic graph and $\mathcal{I} = (G,s,t,c,Q)$  an instance of the AQSPP.
We construct the auxiliary graph $G' = (V',A')$ from  $G=(V,A)$ in the following way:
\begin{equation} \label{VcAc}
V' := \{V_{(s,s)},V_{(t,t)}\} \cup \{ V_{e} \;|\; e \in A  \}, \\ \quad
A' := \{(V_{(i,j)},V_{(j,l)})  \;|\;   i\neq l \},
\end{equation}
where $V_{(s,s)}$ and $V_{(t,t)}$ represent vertices $s$ and $t$, respectively.
The costs of the arcs  in the graph $G'$  are given as follows
$$c'_{(V_{e},V_{f})} = \begin{cases}
c_{f} & \text{ if } e = (s,s) \\
0 & \text{ if } f=(t,t)\\
c_{f} + 2 q_{e,f} & \text{ otherwise}. \\
\end{cases}$$
Now, the auxiliary  instance $\mathcal{I}'$ of $\mathcal{I}$ is the following SPP instance $\mathcal{I}' = (G',V_{(s,s)},V_{(t,t)},c')$.

The following theorem  shows that  the optimal $s$-$t$ path for the AQSPP instance $\mathcal{I}$ on a directed acyclic graph
can be obtained by solving  the SPP instance $\mathcal{I}'$.
\begin{theorem}\label{radjacentDAGtheorem}
Let $G$ be a directed acyclic graph,  $\mathcal{I} = (G,s,t,c,Q)$ an AQSPP instance,
and $\mathcal{I}' = (G',V_{(s,s)},V_{(t,t)},c')$ the auxiliary  instance of $\mathcal{I}$.
Then, an optimal solution of $\mathcal{I}$ can be obtained by solving the SPP for $\mathcal{I}'$.
\end{theorem}

\begin{proof} (See also \cite{frey2015quadratic}.)
We first show that  any $s$-$t$ path in $G$ corresponds to a $V_{(s,s)}$-$V_{(t,t)}$ path in $G'$, and vice-versa.
For ease of notation we set $v_1:=s$ and $v_k:=t$.
Let $P = (v_{1},v_{2},\ldots,v_{k})$ be a $v_{1}$-$v_{k}$ path in $G$.
Then it is not difficult to verify that
\[
{P'= (V_{(v_{1},v_{1})},V_{(v_{1},v_{2})},V_{(v_{2},v_{3})},\ldots,V_{(v_{k-1},v_{k})},V_{(v_{k},v_{k})})}
\]
is a $V_{(v_{1},v_{1})}$-$V_{(v_{k},v_{k})}$ path in the graph $G'$.
The cost of the path $P'$ is given by
$\sum_{i=1}^{k-1}c_{(v_{i},v_{i+1})} + 2\cdot\sum_{i=1}^{k-2} q_{(v_{i},v_{i+1}),(v_{i+1},v_{i+2})}$, which is exactly the  cost of the path $P$.

Conversely, let $P' = (V_{(v_{1},v_{1})},V_{(v_{1},v_{2})},V_{(v_{2},v_{3})},\ldots,V_{(v_{k-1},v_{k})},V_{(v_{k},v_{k})})$ be a $V_{(v_{1},v_{1})}$-$V_{(v_{k},v_{k})}$ path  in $G'$.
Take the ordered set of vertices $P = (v_{1},v_{2},\ldots,v_{k})$.
Let us verify that $P$ is a walk that does not contain repeated vertices.
From the definition of $V'$ and $A'$, see \eqref{VcAc}, it follows that  $v_{i} \in V$ for all $i$, and
$(v_{i},v_{i+1}) \in A$ for ${i = 1,\ldots,k-1}$.
It remains now to verify that there do not exist $k,l$ ($k \neq l$) for which  $v_{k} = v_{l}$.
Indeed, since $G$ is acyclic this is not possible.
Thus $P$ is a $s$-$t$ path in $G$ whose total cost equals to the linear cost of $P'$.
\end{proof}
Note that if $G$ is not acyclic,  then there may exist a $V_{(s,s)}$-$V_{(t,t)}$ path in the
auxiliary graph for which does not exist  a corresponding $s$-$t$ path in $G$.
Let us give an example.

\begin{example}{\rm
Consider a QSPP instance on the  directed graph $G$ from Figure $\ref{AQSPPpicture1}$.
The costs are given as follows. Set $c_{(3,4)} = \epsilon$ for some $0<\epsilon<1$ and ${q_{(1,2),(2,5)}=1}$.
All other linear and interaction costs are zero.
Set for the source and target vertex  $s = 1$ and $t = 5$, respectively.
Clearly, we have a well defined AQSPP instance.
Moreover,  $P=(1,2,5)$  is the unique $s$-$t$ path in $G$,  whose cost is two.
\begin{figure}[H]
	\centering
	\begin{tikzpicture}
	[scale=.8,auto=left,every node/.style={draw, circle}]
	\node (n1) at (0,0) {1};
	\node (n2) at (3,0)  {2};
	\node (n3) at (4.5,3) {3};
	\node (n4) at (1.5,3)  {4};
	\node (n5) at (6,0)  {5};

	\foreach \from/\to in
	{n1/n2,n2/n3,n3/n4,n4/n2,n2/n5}
	\draw [->] (\from) -- (\to);	
	\end{tikzpicture}
	 \caption{Example graph $G$.}
	 \label{AQSPPpicture1}
\end{figure}
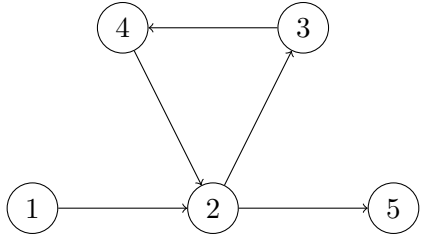

We construct the graph $G'$ from $G$, see Figure $\ref{AQSPPpicture2}$.
It is not difficult to verify that
$(V_{(1,1)},V_{(1,2)},V_{(2,3)},V_{(3,4)},V_{(4,2)},V_{(2,5)},V_{(5,5)})$ is \textcolor{black}{the shortest} $V_{(1,1)}$-$V_{(5,5)}$ path in $G'$, whose cost is $\epsilon$.
However this $V_{(1,1)}$-$V_{(5,5)}$ path does not correspond to a path in $G$, \textcolor{black}{but to a walk}. }

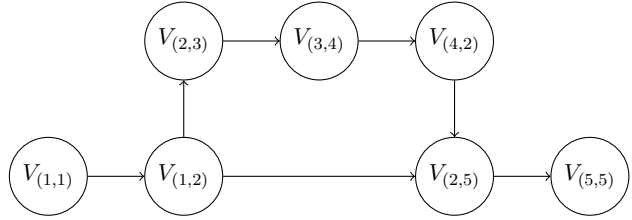
\begin{figure}[H]
	\centering
	\begin{tikzpicture}
	[scale=.6,auto=left,every node/.style={draw, circle,scale=0.8}]
	\node (n1) at (0,0) {$V_{(1,1)}$};
	\node (n2) at (3,0)  {$V_{(1,2)}$};
	\node (n3) at (3,3) {$V_{(2,3)}$};
	\node (n4) at (6,3)  {$V_{(3,4)}$};
	\node (n5) at (9,3)  {$V_{(4,2)}$};
	\node (n6) at (9,0)  {$V_{(2,5)}$};
	\node (n7) at (12,0)  {$V_{(5,5)}$};
	\foreach \from/\to in
	{n1/n2,n2/n3,n3/n4,n4/n5,n5/n6,n6/n7,n2/n6}
	\draw [->] (\from) -- (\to);	
	\end{tikzpicture} \caption{\textcolor{black}{The auxiliary graph $G'$ of $G$.}}
	\label{AQSPPpicture2}
\end{figure}
\end{example}

\subsection{The general adjacent QSPPs}  \label{section:AQSPP_disjoint}

In this section, we prove that the adjacent QSPP that is not restricted to DAGs cannot be approximated unless P=NP.
In particular, we show that the $2$-arc-disjoint paths problem polynomially transforms to the AQSPP.
{\color{black} Rostami} et al.~\cite{rostami2016quadratic} provide a polynomial-time  reduction from  3SAT to the AQSPP.
 Our reduction is considerably shorter (and simpler) than the reduction from \cite{rostami2016quadratic}.

The $k$-arc-disjoint paths problem is defined as follows:
Let $G =(V,A)$ be a directed graph and $(s_{1},t_{1}),\ldots,(s_{k},t_{k})$ pairs of vertices in $G$.
The $k$-arc-disjoint paths problem asks for
pairwise arc-disjoint paths $P_{1},\ldots,P_{k}$ where $P_{i}$ is a $s_{i}$-$t_{i}$ path $(i=1,\ldots,k)$.

An instance of the  $k$-arc-disjoint paths problem can be specified via the tuple $\mathcal{I} = (G,(s_{1},t_{1}),\ldots,(s_{k},t_{k}))$.
Fortune et al.~\cite{fortune1980directed} prove that the $k$-arcs-disjoint paths problem is NP-complete even for $k =2$.
We use this  to prove the main result in this section.
\begin{theorem} \label{AQSPP-NPhard}
The adjacent QSPP on a cyclic digraph cannot be approximated \textcolor{black}{within a constant factor} unless P=NP.
\end{theorem}
\begin{proof}
We provide a polynomial-time reduction from the 2-arc-disjoint paths problem to the adjacent QSPP.
Let
$\mathcal{I} = (G,(s,t),(\bar{s},\bar{t}))$ such that $s \neq \bar{s}$ and $t \neq \bar{t}$ be an instance of the
$2$-arc-disjoint paths problem on the graph $G=(V,A)$.

We construct a directed graph $G'=(V',A')$, where  the vertex set is given as follows
\[
V' = \left \{ v^{1}, v^{2} \;|\;  v\in V \} \cup \{N_{uv} \;|\;  (u,v) \in A \right \}.
\]
In particular,   for each vertex $v \in V$  there are two vertices $ v^{1}, v^{2}$ in $V'$, and for each arc $(u,v) \in A$ there is  the  vertex $N_{uv}\in V'$.
The arc set $A'$ is given by
\[
A' = \left  \{ (u^{i},N_{uv}),\; (N_{uv},v^{i})  \;|\;  (u,v) \in A, \;  i = 1,2 \} \cup \{(t^{1},\bar{s}^{2}) \right \},
\]
i.e., for each arc $(u,v)\in A$ there are two pairs of arcs $(u^{i},N_{uv})$, $(N_{uv},v^{i})$ ${(i=1,2)}$, and  additionally the arc $(t^{1},\bar{s}^{2})$.
	
Now we define the  cost functions $c' : A' \rightarrow \R_+$ and $Q':  A'\times A' \rightarrow \R_+$, $Q'=(q_{ef}')$ as follows.
The linear cost of each arc in $A'$ is zero, i.e., $c'(e) = 0$ for every $e \in A'$.
For every two pairs of arcs $(u^{i},N_{uv})$ and $(N_{uv},v^{i})$, ${(i=1,2)}$ the interaction cost is:
\[
q'_{(u^{i},N_{uv})(N_{uv},v^{j}) } = \begin{cases}
0 & \text{ if } i = j  \\
1 & \text{ if } i \neq j.  \\
\end{cases}
\]
The interaction costs of all other pairs of arcs are zero. Clearly, the non-adjacent arcs have zero interaction costs.
Now, the constructed AQSPP instance is ${\mathcal{I}' = (G',s^{1},\bar{t}^{2},c',Q')}$.

It remains to show that $\mathcal{I}$ is a yes-instance  of the $2$-arc-disjoint paths problem on $G$
if and only if $\mathcal{I}'$ has a $s^{1}$-$\bar{t}^{2}$ path of cost zero.
Now if $\mathcal{I}$ is a yes-instance, then there exists a $s$-$t$ path $P_{1} = (v_{1},v_{2},\ldots,v_{k})$
of length, say, $k-1$ with  $v_1:=s$ and $v_{k} :=t,$ and
a $\bar{s}$-$\bar{t}$ path $P_{2} = (u_{1},u_{2},\ldots,u_{l})$
of length, say, $\l-1$ with  $u_1:=\bar{s}$ and $u_{l} :=\bar{t}$, such that $P_{1}$ and $P_{2}$ are arc-disjoint.
Clearly,  the following ordered set of vertices
\begin{equation}\label{AQSPPcomplexityP}
\begin{aligned}
P' =  (& v_{1}^{1}, \; N_{v_{1},v_{2}}, \;v_{2}^{1}, \;N_{v_{2},v_{3}}, \;v_{3}^{1}, \ldots, v_{k-1}^{1}, \;N_{v_{k-1},v_{k}}, \;v_{k}^{1}, \\
& u_{1}^{2}, \;N_{u_{1},u_{2}}, \;u_{2}^{2}, \;N_{u_{2},u_{3}}, \;u_{3}^{2},\ldots,u_{l-1}^{2}, \;N_{u_{l-1},u_{l}}, \;u_{l}^{2})
\end{aligned}
\end{equation}
is a $v_{1}^{1}$-$u_{l}^{2}$ path of cost zero.

Conversely, a $s^{1}$-$\bar{t}^{2}$ path $P'$ in $G'$ of cost zero
\textcolor{black}{has to consist of a sequence of vertices with superindex one followed by a sequence of vertices with superindex two, as specified in  \eqref{AQSPPcomplexityP}.}
We take the following ordered sets of vertices $P_{1} = (v_{1},v_{2},\ldots,v_{k})$ and $P_{2} = (u_{1},u_{2},\ldots,u_{l})$ in $G$.
It is not difficult to verify that $P_{1}$ and $P_{2}$  are two paths.
It remains to show that $P_{1}$ and $P_{2}$ are arc-disjoint.
Let us assume this is not the case.
Then  there exists  an arc  $(q,w) \in A$ visited by both paths $P_1$ and $P_2$.
From the construction of $P'$, it follows that the arcs $(q^1,N_{q,w})$, $(N_{q,w},w^1)$, $(q^2,N_{q,w})$ and $(N_{q,w},w^2)$
are visited by the path $P'$. This means that the vertex $N_{q,w}$ is visited twice in the $s^{1}$-$\bar{t}^{2}$ path $P'$ in $G'$,
which contradicts to the fact that $P'$ is a walk that does not contain repeated vertices.

Finally, the inapproximability result follows since the objective value of any feasible solution of the constructed AQSPP is either zero or at least two.
\end{proof}

\section{Polynomially solvable cases of the QSPP }\label{section:linearQSPP}

In this section we investigate polynomially solvable cases of the quadratic shortest path problem.
Here we prove, among other things, that all QSPP instances on a directed cycle can be solved by solving an appropriate instance of the SPP.
In Section \ref{sec:costMatr} we consider special cost matrices such as sum and product \textcolor{black}{matrices}.

An instance of the QSPP  is said to be \emph{linearizable}
if there exists a corresponding instance of the SPP {\color{black}with the cost vector  $c'\geq 0$} such that associated  costs are equal i.e.,
\[
C(P,c,Q) = C(P,c'),
\]
for every $s$-$t$ path $P$ in $G$.
 We call $c'$ \textcolor{black}{a} {\em linearization vector} of the QSPP instance.
Linearizable instances for the quadratic assignment problem were considered in e.g., \cite{adams2014linear,cela2014linearizable},
and linearizable instances  for the quadratic minimum spanning tree problem in \cite{custic2015characterization}.

We start this section by proving several basic results.
\begin{lemma} \label{linear:czero}
If the QSPP instance ${\mathcal{I} = (G,s,t,c,Q)}$ is linearizable, and $d$ is a vector such that $c+d \geq 0$,
then the QSPP instance ${\mathcal{I}' = (G,s,t,c+d,Q)}$ is also linearizable.
\end{lemma}
\begin{proof}
Since $\mathcal{I}$ is linearizable, there exists a linear cost vector $c'$ such that
$\sum_{e,f \in P} q_{ef} + \sum_{e\in P} c_{e}  =  \sum_{e\in P} c_{e}'$ for every $s$-$t$ path $P$ in $G$.
Let $c'' := c' + d$, then we have
\begin{align*}
C(P,c+d,Q)=\sum_{e,f \in P} q_{ef} + \sum_{e\in P} (c_{e}+d_{e}) = \sum_{e\in P} c_{e}' +  d_{e} = \sum_{e\in P} c_{e}'' = C(P,c'')
\end{align*}
for every $s$-$t$ path $P$. Thus $\mathcal{I}'$ is also linearizable.
\end{proof}

\begin{lemma}
If  two QSPP instance ${\mathcal{I}_1 = (G,s,t,c_1,Q_1)}$ and ${\mathcal{I}_2 = (G,s,t,c_2,Q_2)}$ are  linearizable,
then the QSPP instance ${\mathcal{I}_3 = (G,s,t,\alpha_1 c_1+ \alpha_2c_2, \alpha_1 Q_1 + \alpha_2 Q_2)}$ is also
linearizable for all nonnegative scalars $\alpha_1$, $\alpha_2$.
\end{lemma}
\begin{proof} Similar to the proof of Lemma \ref{linear:czero}.
\end{proof}

An instance of the QSPP  may be linearizable if the underlying graph has special structure and/or the corresponding quadratic cost matrix has special properties.
Let us give a class of the  QSPP instances that is linearizable for any pair of cost matrices  $(Q,c)$.
The directed cycle  $C_{n}^*$ of order $n$ is a graph with the vertex set $ \{v_{1},\ldots,v_{n}\}$
and arc set $ \{(v_{i},v_{i+1}) \;|\; i = 1,\ldots,n\}$ where addition is modulo $n$.
{\color{black} It is not difficult to verify that any  QSPP instance  on $C_{n}^*$ is linearizable. }

Directed cycles are not the only digraphs on  which QSPP instances are linearizable.
However, they seem to be  easiest cases.
In the following sections we show \textcolor{black}{necessary conditions for which a QSPP instance on a directed complete graph is linearizable.
We also show that every instance on a tournament with  four vertices is linearizable.}

\subsection{Special cost matrices} \label{sec:costMatr}

If the interaction cost matrix has a special structure, then the associated QSPP instance may be solved efficiently.
We consider here  two types of cost matrices for which QSPP instances are linearizable.

We say that a matrix $M \in \R^{m\times n}$ is a \emph{sum matrix} generated by vectors
$a \in  \R^{m}$ and $b \in  \R^{n}$ if $M_{i,j} = a_{i} + b_{j}$ for every $i = 1,\ldots,m$ and $j= 1,\ldots,n$.
A matrix is called a \emph{weak sum matrix} if the condition above is not required for the diagonal entries.
It is not difficult to show that  if a sum matrix $M$ of order $n$ is symmetric, then there exists a vector
$a \in  \R^{n}$ such that  $M_{i,j} = a_{i} + a_{j}$ for all $i,j = 1,\ldots,n$.
Similarly, if the weak sum matrix $M$ of order $n$ is symmetric, then there exists a vector  $a \in  \R^{n}$
such that $M_{i,j} = a_{i} + a_{j}$ for all $i,j = 1,\ldots,n, \; i\neq j$.
Recognition of a (weak) sum matrix can be done efficiently.
Sum matrices are also considered in the context of the QAP \cite{cela2014linearizable},
and the quadratic minimum spanning tree problem \cite{custic2015characterization}.

The following result shows that  symmetric weak sum matrices provide linearizable QSPP instances
\textcolor{black}{on graphs where all $s$-$t$ paths have the same length.}
\begin{proposition} \label{prop:weakSum}
Let $\mathcal{I} = (G,s,t,c,Q)$ be a QSPP instance.
If every $s$-$t$ path in $G$ has the same length and $Q$ is a symmetric weak sum matrix, then $\mathcal{I}$ is linearizable.
\end{proposition}
\begin{proof}
Suppose that every $s$-$t$ path in $G$ has length $L$.
Since $Q$ is a symmetric weak sum matrix,  there exists a vector  $a \in  \R^{m}$ such that
$q_{e,f} = a_{e} + a_{f}$ for all $e,f = 1,\ldots,m, \; e\neq f$.
Let $P$ be a $s$-$t$ path in $G$ with the arc set  $\{ e_i :~i=1,\ldots,L \}$.
Then the cost of $P$ is:
\begin{align*}
C(P,c,Q) & = \sum_{i =1}^{L} \sum_{j =1}^{L} q_{e_{i},e_{j}} + \sum_{i =1}^{L} c_{e_{i}} = \sum_{i =1}^{L} \sum_{j =1, i \neq j }^{L} (a_{e_{i}}+a_{e_{j}})  + \sum_{i =1}^{L} c_{e_{i}} \\
& =  2 (L-1) \sum_{i =1}^{L} a_{e_{i}}  + \sum_{i =1}^{L} c_{e_{i}}=\sum_{i =1}^{L}  2 (L-1) a_{e_{i}} + c_{e_{i}}.
\end{align*}
Now,  define the linear cost $c'_{e} := 2(L-1)a_{e} + c_{e}$ for every arc $e=1,\ldots,m$.
Thus, $C(P,c,Q) = C(P,c')$ for every $s$-$t$ path $P$ in $G$.
\end{proof}

Let us give two examples of graphs whose all $s$-$t$ paths have constant length.
\begin{example} \label{def:gridGraph}{\rm
The \emph{directed grid graph} $G_{p,q} = (V,A)$ is defined as follows. The set of vertices and the set of arcs are given as follows:
\begin{equation} \label{GppAv}
\begin{array}{ll}
V =& \left \{ v_{i,j} \;|\; 1 \leq i \leq p, \; 1 \leq j \leq q \right \}, \\[1ex]
A =&  \left  \{ (v_{i,j},v_{i+1,j}) \;|\; 1 \leq i \leq p - 1, \; 1 \leq j \leq q  \right \} \\[1ex]
 &\cup   \left  \{ (v_{i,j},v_{i,j+1}) \;|\; 1 \leq i \leq p, \; 1 \leq j \leq q-1  \right \}.
\end{array}
\end{equation}
Note that $|V| = pq$ and $|A| = 2pq-p-q$.
If $s = (1,1)$ and $t = (p,q)$, then every $s$-$t$ path has length $p+q-2$. }
\end{example}

\begin{example}{\rm
The \emph{directed hypercube graph} $H_{n}$ is defined as follows. There is a vertex for each binary string of length $n$, there is an arc $(u,v)$
if the vertices $u$ and $v$ differ in exactly one bit position and the binary value of $u$ is less than the binary value of $v$.
Note that $H_{n}$ has $2^{n}$ vertices, $ 2^{n-1}n$ arcs.
If $s$ is an all-zeros string and $t$ is an all-ones string, then every $s$-$t$ path has length $n$.}
\end{example}

Let us consider  another special cost matrix. A matrix $M \in \R^{m\times m}$ is called a \emph{symmetric product matrix} if $M = aa^{\mathrm T}$ for some vector $a \in  \R^{m}$.
A nonnegative product matrix with integer values plays a role in the Wiener maximum  QAP, see \cite{ccela2011wiener}.

If $Q+\text{Diag}(c)$ is a positive semidefinite matrix, then the tuple $(G,s,t,c,Q)$ is a convex QSPP instance.
Here, the 'Diag' operator maps a $m$-vector to the diagonal matrix, by placing the vector  on the diagonal of the $m\times m$ matrix.
In \cite{rostami2016quadratic}, it is shown that the convex QSPP is APX-hard,
but  can be approximated within a factor of $|V|$. The following result shows that one can solve the QSPP efficiently whenever $Q+\text{Diag}(c)$
is a symmetric product matrix, thus positive semidefinite matrix of rank one.
\begin{proposition} \label{prop:productmatrix}
Let $\mathcal{I} = (G,s,t,c,Q)$ be a QSPP instance.
If $Q + {\rm Diag}(c)$ is a nonnegative symmetric product matrix,
then $\mathcal{I}$ is solvable in ${\mathcal O}(m+n\log n)$ time.
\end{proposition}
\begin{proof}
Since $Q + \text{Diag}(c)$ is a nonnegative symmetric product matrix,
there exists a vector  $a \in  \R^{m}_{+}$ such that $Q + {\rm Diag}(c) = aa^{\mathrm T}$.
Let $x \in \{0,1\}^{m}$  be the characteristic vector of a $s$-$t$ path $P$ in $G$.
The cost of this path satisfies
\[
C(P,c,Q) = x^{\mathrm T}(Q + \text{Diag}(c))x = x^{\mathrm  T}aa^{\mathrm  T}x = (x^{\mathrm T}a)^2.
\]
Let us define the linear cost vector $c'\in  \R^{m}_{+}$ by taking $c'_{e} = a_{e}$ for every $e$.
Thus, we have $C(P,c,Q) = C(P,c')^2$ for every $s$-$t$ path $P$ in $G$, and the complexity of solving
 the shortest path problem by Dijkstra's algorithm is  ${\mathcal O}(m+n\log n)$.
\end{proof}

\section{The QSPP on  complete digraphs}\label{sect:completeGraphs}

In this section we analyze the QSPP on the complete symmetric digraph $K_n^*$, that is a digraph in which every pair of vertices
is connected by a bidirectional edge.  It is trivial to solve the QSPP on $K_{n}^*$ for $n\leq 3$.
Here, we provide sufficient and necessary conditions for a QSPP instance on $K_4^*$ to be linearizable.
We also provide several properties of linearizable  QSPP instances on  $K_n^*$ with $n \geq 5$.

{\bf Assumptions.} In this section we assume that the following trivial arcs are removed from $K_n^*$:
the incoming arcs to $s$, outgoing arcs from $t$, and the arc $(s,t)$.
For example, the removal of these arcs from $K_{4}^*$ results in the simplified graph, see Figure $\ref{linearization}$.
Note that  removing  the arc $(s,t)$ does not change the linearizability of an instance.
 Further, we  assume  that  the cost vector $c$ is {\color{black} the all-zero vector} (see Lemma \ref{linear:czero}), and that
 interaction costs of pairs of arcs that can not be  together included in any $s$-$t$ path is zero.
 For example, the interaction cost of arcs $(v_1,v_2)$ and $(v_3,v_2)$ in graph from Figure $\ref{linearization}$ is zero.
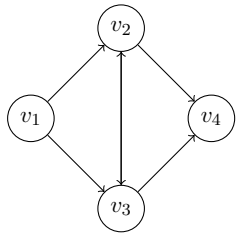
\begin{figure}[H]
	\centering
	\begin{tikzpicture}
	[scale=.6,auto=left,every node/.style={draw, circle,scale=0.8}]
	\node (n1) at (0,0) {$v_{1}$};
	\node (n2) at (2,2)  {$v_{2}$};
	\node (n3) at (2,-2) {$v_{3}$};
	\node (n4) at (4,0)  {$v_{4}$};
	\foreach \from/\to in
	{n1/n2,n1/n3,n2/n3,n3/n2,n3/n4,n2/n4}
	\draw [->] (\from) -- (\to);	
	\end{tikzpicture} \caption{Simplified $K_{4}^*$ with $s=v_{1},$ $t= v_{4}$.}
	\label{linearization}
\end{figure}
Let ${\mathcal{I} = (K_{n}^*,s,t,c,Q)}$ ($n\geq 4$) be an instance of the QSPP.
We classify  $s$-$t$ paths  by their lengths for that instance.
This classification leads us to necessary and/or sufficient conditions for a QSPP instance to be linearizable.

Let $\mathcal{P}_{k}$ denotes the set of $s$-$t$ paths of length $k$ for $k \in \{2,\ldots,n-1\}$.
The total cost of all $s$-$t$ paths of length $k$ is  $\mathcal{CP}_{k} = \sum_{P \in \mathcal{P}_{k}}C(P,c,Q)$.
The number of $s$-$t$ paths of the length $k$ is $|\mathcal{P}_{k}| = {n-2 \choose k-1}\cdot (k-1)!$.
In what follows, we  show that $\mathcal{CP}_{k}$ is bounded from above by $\mathcal{CP}_{k+1}$ for  $k=3,\ldots, n-2$, and several other related results.

\begin{proposition}\label{linear:complete_average}
Let  ${\mathcal{I} = (K_{n}^*,s,t,c,Q)}$ be a QSPP instance and $n \geq 4$.
Then the average cost of all $s$-$t$ paths of the length $k$ is not greater than the average cost of all $s$-$t$ paths of the length $k+1$, i.e.,
$$\frac{1}{|\mathcal{P}_{k}|}\mathcal{CP}_{k} \leq \frac{1}{|\mathcal{P}_{k+1}|}\mathcal{CP}_{k+1},$$
for $k = 3,\ldots,n-2.$
\end{proposition}
\begin{proof}
We will derive an expression for $\frac{1}{|\mathcal{P}_{k}|}\mathcal{CP}_{k}$ ($k \geq 2$) in terms of the interaction costs.
Then, the claim follows from the fact that the expression is an increasing function for $k \geq 3$.
	
Given an arc $e = (i,j)$, we define $h(e) = i$ to be the head vertex $i$, and $t(e) = j$ to be the tail vertex  $j$. Let
\begin{equation} \label{linear:complete_average_H}
H = \{ e \in A \;|\; h(e) = s \text{ or } t(e) = t\}
\end{equation}
be the set of arcs either leaving $s$ or entering $t$.
Let $S = \left \{ \{e,f\} \;|\;  t(e) = h(f) \text{ or } h(e) = t(f) \right \}$ be the set of  distinct unordered pairs of adjacent arcs.
Based on the sets $H$ and $S$, we define the following sets of distinct unordered pairs of arcs:
\begin{align*}
T_{1} = \{\{e,f\} \in S  \;|\; e\in H \text{ and } f \in H \}, \quad T_{2} = \{\{e,f\} \notin S  \;|\; e\in H \text{ and } f \in H \},\\
T_{3} = \{\{e,f\}  \in S \;|\; e\in H \text{ and } f \notin H \}, \quad T_{4} = \{\{e,f\}  \notin S \;|\; e\in H \text{ and } f \notin H \},\\
T_{5} = \{\{e,f\} \in S  \;|\; e\notin H \text{ and } f \notin H \}, \quad T_{6} = \{\{e,f\} \notin S  \;|\; e\notin H \text{ and } f \notin H \}.
\end{align*}
Clearly,  $H$ and its complement partition the arc set, and $T_{1},\ldots,T_{6}$ partition the arc pairs in $K_{n}^*$.
The sum of interaction costs over the pairs of arcs in $T_{i}$ is\\ ${s_{i} = 2\cdot \sum_{\{e,f\}\in T_{i}}  q_{ef}}$ for $i = 1,\ldots,6$.
Note that ${\sum_{i=1}^{6}s_{i}= u^{\mathrm T}Qu}$, where {\color{black}$u \in \R^m$} is the vector of all-ones.

It holds that every  pair of arcs $\{e,f\} \in T_{i}$  ($i \in \{1,\ldots,6\}$),
which is included in at least one $s$-$t$ path,
is contained in the same number, denoted by $t_{i,k}$, of $s$-$t$ paths  of length $k$.
In particular we have
\begin{align*}
&t_{1,k} = \begin{cases}
1 & \text{ if } k = 2 \\
0 & \text{ otherwise},
\end{cases} & \qquad
t_{2,k} = \begin{cases}
{n-4 \choose k-3}\cdot (k-3)! & \text{ if } k \geq 3 \\
0 & \text{ otherwise},
\end{cases} \\
&t_{3,k} = \begin{cases}
{n-4 \choose k-3}\cdot (k-3)! & \text{ if } k \geq 3 \\
0 & \text{ otherwise},
\end{cases}
& \qquad
t_{4,k} = \begin{cases}
{n-5 \choose k-4}\cdot (k-3)! & \text{ if } k \geq 4 \\
0 & \text{ otherwise},
\end{cases}\\
&t_{5,k} = \begin{cases}
{n-5 \choose k-4}\cdot (k-3)!  & \text{ if } k \geq 4 \\
0 & \text{ otherwise},
\end{cases}
& \qquad t_{6,k} = \begin{cases}
{n-6 \choose k-5}\cdot (k-3)!  & \text{ if } k \geq 5 \\
0 & \text{ otherwise}.
\end{cases}
\end{align*}
For $k \geq 2$, the average cost of all $s$-$t$ paths of length $k$ can be written as:
\begin{align}
& \frac{1}{|\mathcal{P}_{k}|}\mathcal{CP}_{k}  = \frac{1}{|\mathcal{P}_{k}|} \sum_{i=1}^{6}t_{i,k} \cdot s_{i} \label{linear:complete_averagepathcost}\\
& = \frac{1}{n-2} \cdot s_{1}\cdot  \mathbbm{1}_{\{k=2\}} + \frac{1}{(n-2)(n-3)}\cdot (s_{2}+s_{3} )\cdot \mathbbm{1}_{\{k\geq 3\}}  \nonumber \\
& + \frac{(k-3)}{(n-2)(n-3)(n-4)}\cdot (s_{4}+s_{5}) \cdot \mathbbm{1}_{\{k\geq 4\}}  +\frac{(k-3)(k-4)}{(n-2)(n-3)(n-4)(n-5)}\cdot s_{6} \cdot \mathbbm{1}_{\{k\geq 5\}}.  \nonumber
\end{align}
Here, $\mathbbm{1}_{A}$ is the indicator function defined as $\mathbbm{1}_{x} = 1$ if condition $x$ is true, and zero otherwise.
It is clear that $\frac{1}{|\mathcal{P}_{k}|}\mathcal{CP}_{k}$ is an increasing function in $k \geq 3$  and this finishes the proof.
\end{proof}
Note that from \eqref{linear:complete_averagepathcost} it follows that one can easily compute $\mathcal{CP}_{k}$ for $k\geq 2$.
As  direct consequences of the previous proposition, we have the following two results.
\begin{corollary} \label{linear:complete_upper}
\begin{enumerate}
\item[a)] Let  ${\mathcal{I} = (K_{n}^*,s,t,c,Q)}$ be a QSPP instance and $n \geq 4$. Then
$$\mathcal{CP}_{k} \leq \frac{1}{n-k-1} \mathcal{CP}_{k+1},$$
for $k = 3,\ldots,n-2.$

\item[b)] Let  ${\mathcal{I} = (K_{n}^*,s,t,c,Q)}$ be a QSPP instance and $n \geq 7$. Then
$$ \mathcal{CP}_{k} \leq (n-k)\cdot \frac{k-3}{k-5} \cdot \mathcal{CP}_{k-1}$$
for $k = 6,\ldots,n-1.$

\end{enumerate}
\end{corollary}

\noindent {\em Proof.}
The first part follows directly from Proposition \ref{linear:complete_average}.
To show the second part, note that $t_{i,k}$ in the proof of Proposition \ref{linear:complete_average} satisfy
$\frac{t_{2,k}}{t_{2,k-1}} = \frac{t_{3,k}}{t_{3,k-1}}  = n-k$, $\frac{t_{4,k}}{t_{4,k-1}} = \frac{t_{5,k}}{t_{5,k-1}} = (n-k)\frac{k-3}{k-4}$,
$\frac{t_{6,k}}{t_{6,k-1}} = (n-k) \frac{k-3}{k-5}$ for $k = 6,\ldots,n-1.$
Thus, $\frac{t_{i,k}}{t_{i,k-1}} \leq (n-k) \frac{k-3}{k-5} $ from where it follows:
\begin{align*}
\mathcal{CP}_{k}  &=  \sum_{i=2}^{6}t_{i,k} \cdot s_{i} =\sum_{i=2}^{6}\frac{t_{i,k}}{t_{i,k-1}} \cdot t_{i,k-1}\cdot s_{i}    \\
& \leq (n-k)\cdot \frac{k-3}{k-5} \cdot \sum_{i=2}^{6} t_{i,k-1}\cdot s_{i} = (n-k)\cdot \frac{k-3}{k-5}  \cdot \mathcal{CP}_{k-1}.  \qquad\qquad\qquad \qed
\end{align*}
From Corollary \ref{linear:complete_upper}, it follows that  $\mathcal{CP}_{2}$ is not bounded by $\mathcal{CP}_{k}$ for $k \geq 3$.
This is due to the fact that the interaction costs of arc pairs in $T_{1}$ only contribute to $\mathcal{CP}_{2}$.
\textcolor{black}{In particular, $\mathcal{CP}_{2}= t_{1,2}s_{1}$ and $\mathcal{CP}_{3}= t_{2,3}s_{2}+t_{3,3}s_{3}$.}
The second inequality in Corollary \ref{linear:complete_upper} shows that
$\mathcal{CP}_{k}$  for $k = 4,5$ can be arbitrarily bigger than $\mathcal{CP}_{k-1}$.
This is because the interaction costs of pairs in $T_{4}$ and $T_{5}$ (respectively $T_{6}$) only contribute to the costs of paths of length greater or equal to four (respectively five).
\textcolor{black}{In particular,  $\mathcal{CP}_{4}= t_{2,4}s_{2}+t_{3,4}s_{3} + t_{4,4}s_{4}+t_{5,4}s_{5}$ and $\mathcal{CP}_{5}= t_{2,5}s_{2}+t_{3,5}s_{3} + t_{4,5}s_{4}+t_{5,5}s_{5}+t_{6,5}s_{6}$.}

In what follows, we show that linearizability imposes stricter conditions on interaction costs than those  given above.
In particular,  $\mathcal{CP}_{2}$ has to be upper bounded by $\mathcal{CP}_{3}$, and $\mathcal{CP}_{k}$ by $\mathcal{CP}_{k-1}$ for $k \geq 4$
with a constant that is tighter than the one from  Corollary \ref{linear:complete_upper}.

Let us first introduce the {\em path matrix}.
Given a QSPP instance $\mathcal{I} = (G,s,t,c,Q)$, the  $s$-$t$  \emph{path  matrix} $B$ is a matrix whose rows are the characteristic vectors of the $s$-$t$ paths in $G$.
Thus,  the rows and columns of $B$ are indexed by the paths and the arcs of $G$, respectively.
The \emph{cost vector} $b$ is defined as $b_{i} := C(P_{i},c,Q)$.

For example, let $\mathcal{I} = (K_{4}^*,s,t,c,Q)$ be a QSPP instance such that $s=v_1$ and $t=v_4$, see  Figure \ref{linearization}.
The $s$-$t$ path matrix and cost vector are given as follows:
\begin{equation} \label{pathMatr}
B =
\kbordermatrix{ & (v_{1},v_{2}) & (v_{1},v_{3}) & (v_{2},v_{3}) & (v_{3},v_{2}) & (v_{2},v_{4}) & (v_{3},v_{4})\cr
	P_{1} & 1 & 0 & 0 & 0 & 1 & 0 \cr
	P_{2} & 0 & 1 & 0 & 0 & 0 & 1 \cr
	P_{3} & 1 & 0 & 1 & 0 & 0 & 1\cr
	P_{4} & 0 & 1 & 0 & 1 & 1 & 0} \quad \text{ and } \quad b = \begin{bmatrix}
C(P_{1},c,Q)\\ C(P_{2},c,Q) \\ C(P_{3},c,Q) \\ C(P_{4},c,Q)
\end{bmatrix}
\end{equation}

Note that a QSPP instance on $G$ is  linearizable if and only if the linear system  $Bc' = b$, $c' \geq 0$ with  variable $c' \in \R^{m}$ has a solution.
\begin{example} {\em
Let us consider a QSPP instance  $\mathcal{I} = (K_{4}^*,s,t,c,Q)$, see Figure \ref{linearization}.
Let $s=v_1$ and $t=v_4$, and  $q_{(v_{1},v_{2}),(v_{2},v_{4})} = q_{(v_{1},v_{3}),(v_{3},v_{4})}  = 1$, and $q_{e,f} = 0$ for all other pairs of arcs $(e,f)$.
Then the costs satisfy $C(P_{1},c,Q) =  C(P_{2},c,Q) = 2$ and $C(P_{3},c,Q) =  C(P_{4},c,Q) = 0$.
Thus, the sum of the  costs of paths of length two is greater than  the sum of the  costs of paths of length three.
Now, for $y = (-1,-1,1,1)^{\mathrm T}$ we have that  $B^{\mathrm  T}y \geq 0$ and $b^{\mathrm T} y  = -4 < 0$,
and from the Farkas' lemma it follows that the  QSPP instance  $\mathcal{I}$ is not linearizable.}
\end{example}
In the following proposition we derive necessary conditions that should satisfy a linearizable QSPP instance on complete digraph.
 The following conditions require that $\mathcal{CP}_{2}$ is bounded by $\mathcal{CP}_{3}$,
and also $\mathcal{CP}_{k}$  by $\mathcal{CP}_{k-1}$  for $k=4,5$.
Note that those constraints are not imposed in general, see Corollary \ref{linear:complete_upper}.
\begin{proposition} \label{necessary_complete}
\begin{enumerate}
\item[a)] Let  ${\mathcal{I} = (K_{n}^*,s,t,c,Q)}$ be a QSPP instance and  $n \geq 4$.
If $\mathcal{I}$  is linearizable, then
\[
 \mathcal{CP}_{k} \leq \frac{1}{n-k-1} \cdot \mathcal{CP}_{k+1}
\]
for  $k =2,\ldots, n-2$.
\item[b)] Let  ${\mathcal{I} = (K_{n}^*,s,t,c,Q)}$ be a QSPP instance and  $n \geq 5$.
If $\mathcal{I}$  is linearizable, then
\[
 \mathcal{CP}_{k} \leq (n-k) \cdot \frac{k-2}{k-3} \cdot \mathcal{CP}_{k-1}
\]
for every $k =4,\ldots, n-1$.
\end{enumerate}
\end{proposition}
\begin{proof}
Let us first show the first claim. Define $g(k) := {n-3 \choose k-2} (k-2)!$ for $k \geq 2$,  and
\[
g'(k) := \begin{cases}
{n-4 \choose k-3} (k-2)! & \text{ for } k \geq 3, \\
0 & \text{ otherwise}.
\end{cases}.
\]

Let $H$ be the arc set defined in $(\ref{linear:complete_average_H})$.
It is not difficult to see that for every arc $e \in H$ (resp.~$e \notin H$), there are $g(k)$ (resp.~$g'(k)$) $s$-$t$ paths of length $k$ containing $e$.
Note that $g(k+1) = (n-k-1) \cdot g(k)$  and $g'(k+1)  \geq (n-k-1) \cdot g'(k)$ for every $k \geq 2$, as $g'(k+1) = (n-k-1) \cdot  \frac{k-1}{k-2} \cdot g'(k)$ for $k \geq 3$.

Take the vector $y$ such that
$$ y_{i} = \begin{cases}
-(n-k-1) & \text{ if } |P_{i}| = k \\
1 & \text{ if } |P_{i}| = k+1 \\
0 & \text{ otherwise}.
\end{cases}$$
Now, for the path matrix $B$ of $K_n^*$ it follows that  $B^{\mathrm T}y \geq 0$, and for the cost vector $b^{\mathrm T}y = -(n-k-1)\mathcal{CP}_{k} +  \mathcal{CP}_{k+1}$.
If $\mathcal{I}$ is linearizable, then $b^{T}y \geq 0$  from where it follows the first claim \textcolor{black}{(by applying Farkas' lemma).}

In a similar fashion, we prove the second claim by taking
\[
 y_{i} = \begin{cases}
(n-k) \cdot \frac{k-2}{k-3} & \text{ if } |P_{i}| = k-1 \\
-1 & \text{ if } |P_{i}| = k \\
0 & \text{ otherwise}.
\end{cases}
\]
\end{proof}

The results from Proposition \ref{necessary_complete} indicate that appropriate restrictions on $\mathcal{CP}_{k}$
might lead to sufficient conditions for a QSPP instance to be linearizable.
Indeed, the next proposition provides  characterization of linearizable QSPP instances on $K_{4}^*$.
\begin{proposition} \label{linear:K4}
Let $\mathcal{I} = (K_{4}^*,s,t,c,Q)$ be a QSPP instance.
$\mathcal{I}$  is linearizable if and only if
$$C(P_{1},c,Q)+C(P_{2},c,Q) \leq C(P_{3},c,Q)+ C(P_{4},c,Q),$$
where
$P_1$, $P_2$ (resp.~$P_3$, $P_4$) are paths of length two (resp.~three).
\end{proposition}
\begin{proof}
Let us denote by  $\{v_{1},v_{2},v_{3},v_{4}\}$ vertices of $K_{4}^*$, and set  $s=v_1$ and $t=v_4$, see Figure \ref{linearization}.
We denote four paths as follows, $P_{1} = (v_{1},v_{2},v_{4})$, $P_{2} = (v_{1},v_{3},v_{4})$, $P_{3}= (v_{1},v_{2},v_{3},v_{4}),$ and $P_{4} = (v_{1},v_{3},v_{2},v_{4})$.

Assume $C(P_{1},c,Q)+C(P_{2},c,Q) \leq C(P_{3},c,Q)+ C(P_{4},c,Q)$.
We construct a vector $c'$ s.t.~$C(P_i,c,Q)=C(P_i,c')$ for every $i=1,\ldots,4$.

Case 1. If $C(P_{1},c,Q) \leq C(P_{3},c,Q)$ and $C(P_{2},c,Q) \leq C(P_{4},c,Q)$, then we set
$$c'_{e} = \begin{cases}
C(P_{1},c,Q) & \text{ if } e = (v_{1},v_{2}), \\
C(P_{2},c,Q) & \text{ if } e = (v_{1},v_{3}), \\
C(P_{3},c,Q) - C(P_{1},c,Q) & \text{ if } e = (v_{2},v_{3}), \\
C(P_{4},c,Q) - C(P_{2},c,Q) & \text{ if } e = (v_{3},v_{2}),\\
0 & \text{ otherwise.}
\end{cases} $$

Case 2. If $C(P_{1},c,Q) > C(P_{3},c,Q)$, then we set
$$c'_{e} = \begin{cases}
C(P_{3},c,Q)  & \text{ if } e = (v_{1},v_{2}), \\
C(P_{2},c,Q) & \text{ if } e = (v_{1},v_{3}), \\
C(P_{1},c,Q) - C(P_{3},c,Q) & \text{ if } e = (v_{2},v_{4}), \\
C(P_{4},c,Q) + C(P_{3},c,Q) - C(P_{1},c,Q) -  C(P_{2},c,Q) & \text{ if } e = (v_{3},v_{2}), \\
0 & \text{ otherwise.}
\end{cases}$$
The remaining cases can be similarly obtained.
In all these cases, we have $c' \geq 0$ and $C(P_{i},c,Q) = C(P_{i},c')$ for every $i = 1,\ldots,4$. Thus, $\mathcal{I}$  is linearizable.

\textcolor{black}{The converse} follows from  Proposition \ref{necessary_complete}.
\end{proof}
Note that  one can easily verify the inequality from Proposition \ref{linear:K4}.
The following example  shows that conditions from  Proposition \ref{necessary_complete} are not sufficient already for $K_5^*$.
\begin{example}
The inequalities in Proposition \ref{necessary_complete} are not sufficient for a QSPP instance on $K_5^*$ to be linearizable.
Take $\mathcal{I} = (K_{5}^*,s,t,c,Q)$,
with $s=v_1$, $t=v_5$,  $q_{(v_{3},v_{4}),(v_{4},v_{5})}= 1$, and $q_{e,f} = 0$ for all other pairs of arcs $(e,f)$.
The costs of the paths $P = (v_{1},v_{3},v_{4},v_{5})$ and $P' = (v_{1},v_{2},v_{3},v_{4},v_{5})$ are two, and all the other paths have zero costs.
Thus we have $\mathcal{CP}_{4} =\mathcal{CP}_{3}  = 2$ and $\mathcal{CP}_{2}  = 0$.
It is readily to check that the both inequalities in Proposition \ref{necessary_complete}
are satisfied. However this instance is not linearizable.
Namely,  the paths $( v_{1},v_{3},v_{4},v_{2},v_{5} )$ and $( v_{1},v_{4},v_{5} )$ have zero costs.
If there would exists a corresponding instance of the SPP with the cost vector $c'$,
then the  linear costs of the arcs $(v_{1},v_{3}), (v_{3},v_{4}),(v_{4},v_{5})$ should  be zeros.
This leads to $0 = C(P,c')\neq C(P,c,Q) = 2$, which is not possible.
\end{example}

Let us assume now that a complete directed graph under consideration is a tournament,
which is  a graph in which every pair of vertices is connected by a single uniquely directed edge.
Then, all instances of a tournament with four vertices are linearizable.
\begin{proposition} \label{linear:T4}
Let $\mathcal{I} = (T_{4}^*,s,t,c,Q)$ be a QSPP instance on a tournament $T_4^*$.
Then, $\mathcal{I}$  is linearizable.
\end{proposition}
\begin{proof} The proof is similar to the proof of \textcolor{black}{Proposition \ref{linear:K4}}.
\end{proof}
There exist QSPP instances on a tournament with five vertices that are not linearizable.

\section{The QSPP on directed grid graphs} \label{sect:DGG}

Directed grid graphs are introduced in Section \ref{sec:costMatr}.
The directed grid graph $G_{pq}$ ($p,q\geq 2$) has $pq$ vertices and $2pq-p-q$ arcs, given as in \eqref{GppAv}.
Every $s$-$t$ path in  $G_{pq}$ has the same length.
In Section \ref{sec:costMatr} we prove that  an optimal solution of  the QSPP on  $G_{p,q}$,
whose quadratic cost matrix is a symmetric weak sum matrix or a nonnegative symmetric product matrix can be obtained by solving the corresponding SPP.

In this section we present an algorithm that verifies whether a QSPP instance on a directed
grid graph is linearizable, and if it is linearizable the algorithm returns the corresponding  linearization vector.
The complexity of our algorithm is ${\mathcal{O}(p^{3}q^{2}+p^{2}q^{3})}$.
Punnen and Kabadi \cite{punnen2017qap} present an   $\mathcal{O}(n^2)$ algorithm for the  Koopmans-Beckmann QAP linearization problem, where $n$ is the size of the QAP.

In this section we assume that  $s = v_{1,1}$ and $t = v_{p,q}$, unless otherwise specified.
The number of $s$-$t$ paths in $G_{p,q}$ {\color{black}is given as follows.}
\begin{lemma}
The number of $s$-$t$ paths in the directed grid graph $G_{p,q}$ is  ${p+q-2 \choose p -1}$.
\end{lemma}

Let us consider for now only linear costs associated with arcs in $G_{p,q}$.
We say that the linear cost vectors $c$ and $d$ are \emph{equivalent} if $C(P,c) = C(P,d)$ for every $s$-$t$ path in the graph.
Given a linear cost vector $c$ associated with arcs  in $G_{p,q}$ and a vertex $v_{i,j} \notin \{s,t\}$,
we describe how to construct a new linear cost vector $d$, see $\eqref{linear:procedure1}$,
that is equivalent to $c$ and whose associated cost of an outgoing arc from $v_{i,j}$ equals zero.
In particular, let $H  \neq \emptyset$ be the set of incoming arcs to $v_{i,j}$, and $F$ the set of outgoing arcs from $v_{i,j}$.
Let arc $f \in F$  be an outgoing arc from $v_{i,j}$ defined as follows:
\begin{equation}  \label{linear:outgoingArcf}
f = \begin{cases}
(v_{i,j},v_{i,j+1})   & \text{if } j \leq q-1 \\
(v_{i,j},v_{i+1,j})   & \text{if } j = q \\
\end{cases}.
\end{equation}
The new linear cost vector $d$ is given by
\begin{equation}  \label{linear:procedure1}
d_{e}=  \begin{cases}
0 & \text{if } e = f \\
c_{e}  + c_{f} & \text{if } e \in H\\
c_{e}  - c_{f} & \text{if } e \in F \backslash \{f\}\\
c_{e} & \text{otherwise}
\end{cases}.
\end{equation}
In the other words, we redistributed weights of the arcs such that the outgoing arc from $v_{i,j}$ that is  of the form \eqref{linear:outgoingArcf} has zero weight.
One can easily verify that $c$ is equivalent to $d$.
Note that  the shortest path problem on a directed acyclic graph with negative weights remains polynomial-time solvable.

The \emph{depth} of a vertex $v_{i,j}$  ($i=1,\ldots, p,$ $j=1,\ldots, q$) is defined as $i+j$.
If we apply the above procedure to a linear cost vector $c$ repeatedly for each node $v_{i,j} \notin \{s,t\}$
starting with the node whose depth is $p+q-1$ until the node with depth three.
(The order of applying the procedure for the nodes with the same depth is arbitrary.)
Then we obtain a linear cost vector, denoted by $\widehat{c}$,
such that $\widehat{c}_{f} = 0$ for all $f$ given in \eqref{linear:outgoingArcf}.
We say that $\widehat{c}$ is the {\em \textcolor{black}{reduced form}} of $c$, or  $\widehat{c}$  is a linear cost vector in  reduced form.
As an example, the constructed linear cost vector in Lemma \ref{linear:LemmaG2q} is in reduced form.
\textcolor{black}{Note that the reduced form depends on the choice of the outgoing arc $f$.
In what follows, we assume that $f$ is  specified as in  \eqref{linear:outgoingArcf}.}

\begin{lemma}  \label{linear:reducedtime}
If $c$ is a linear cost vector on $G_{p,q}$, then its reduced form $\widehat{c}$ can be computed in $\mathcal{O}(p q)$.
\end{lemma}

The following result shows that there exists \textcolor{black}{a unique} linear cost vector in reduced form.
\begin{lemma} \label{linear:LemmaUniqueReducedF}
If the linear cost vectors $c$ and $d$ on  $G_{p,q}$ are equivalent,
then their reduced forms $\widehat{c}$ and $\widehat{d}$ are equal, i.e., $\widehat{c} = \widehat{d}$.
\end{lemma}
\begin{proof}
For any linear cost vector $c$, its reduced form $\widehat{c}$ satisfies  that $\widehat{c}_{e} = 0$ whenever $e$ does not belong to the set of arcs
$$J = {\{ (v_{i,j},v_{i+1,j}) \;|\; i =1,\ldots,p-1, \; j = 1,\ldots,q-1\} \cup \{(v_{1,1},v_{1,2})\}}.$$
Note that $|J| = (p-1)(q-1) + 1$.
To every arc $e$ from $J$ we assign the path in the following way:
\begin{align*}
P_{e} = \begin{cases}
(v_{1,1},\ldots,v_{1,q},\ldots,v_{p,q}) \!\!\!\!\!\!\!& \text{ if } e = (v_{1,1},v_{1,2}), \\
(v_{1,1},\ldots,v_{1,j},v_{2,j},\ldots,v_{2,q},\ldots,v_{p,q}) \!\!\!\!\!\!\!& \text{ if } e = (v_{1,j},v_{2,j}), j=1,\ldots,q-1  \\
(v_{1,1},\ldots,v_{i,1},\ldots,v_{i,j},v_{i+1,j},\ldots,v_{i+1,q},\ldots,v_{p,q}) \!\!\!\!\!\!\!&
\text{ if } e = (v_{i,j},v_{i+1,j}), i \geq 2\\
&~~~~ j=1,\ldots,q-1.
\end{cases}
\end{align*}
We call those $(p-1)(q-1) + 1$ paths the \emph{critical paths}.
It is not difficult to see that the cost of the critical path $P_e$ for $e \in J$  i.e., $C(P_{e},c)$  uniquely determines the value of
$\widehat{c}_{e}$ for $e \in J$. \textcolor{black}{Moreover}, the reduced form $\widehat{c}$ has $\widehat{c}_{e} = 0$ for $e \notin J$.

Now if $d$ is a linear cost vector that is equivalent to $c$, then $C(P,d) = C(P,c)$ for every $s$-$t$ path $P$ in the graph.
In particular, this equality holds for the critical paths.
This implies that $\widehat{c_{e}} = \widehat{d_{e}}$ for every arc $e \in J$.
Since $\widehat{c_{e}} = \widehat{d_{e}} = 0$ for $e \notin J$, we have $\widehat{c} = \widehat{d}$.
\end{proof}

\medskip
If an instance $\mathcal{I}= (G_{p,q},s,t,c,Q)$  of the QSPP  is linearizable, then all linearizations of $\mathcal{I}$ are equivalent to each other.
Suppose that the vector $c'$ is a linearization vector of $\mathcal{I}$, then the proof of Lemma  \ref{linear:LemmaUniqueReducedF}
gives a recipe to calculate  $\widehat{c'}$, which is the unique linearization vector in reduced form. This recipe uses only the costs of the critical paths to determine $\widehat{c'}$.
Indeed,  the cost of the critical path $P_e$ ($e \in J$) satisfies $C(P_{e},c') = C(P_{e},c,Q)$, where  $c'$  is the linearization vector of $\mathcal{I}$.
The costs $C(P_{e},c,Q)$ for $e \in J$ can be easily obtained from the input instance.

In fact, the above calculation of the unique linear cost vector in reduced form can be implemented
 even \textcolor{black}{if  the linearizability} of $\mathcal{I}$ is not known.
We call the resulting vector the \emph{pseudo-linearization} vector of $\mathcal{I}$, denoted by $\widehat{pc}$.
It is not hard to verify that $\mathcal{I}$ is linearizable if and only if the pseudo-linearization vector $\widehat{pc}$ is a linearization vector of $\mathcal{I}$.

Let us assume from now on that the linear cost vector equals the all-zero vector, i.e., $c=0$.

\begin{lemma} \label{linear:pseudotime}
Let $\mathcal{I}= (G_{p,q},s,t,c,Q)$  be an instance of the QSPP.
The pseudo-linearization vector $\widehat{pc}$ for $\mathcal{I}$ can be computed in $\mathcal{O}(p^{2}q+pq^{2})$ time.
\end{lemma}
\begin{proof}
The quadratic cost of the critical path $P_{(v_{1,1},v_{1,2})} = (v_{1,1},\ldots,v_{1,q},\ldots,v_{p,q} )$ is calculated straightforward via the formula
$2 \sum_{e,f \in P}q_{ef}$ which costs $\mathcal{O}(p^{2}+q^2)$.
The critical path $P_{(v_{1,q-1},v_{2,q-1})} = (v_{1,1},\ldots,v_{1,q-1},v_{2,q-1},v_{2,q},\ldots,v_{p,q})$  differs only in two arcs from $P$.
Thus, its cost can be computed in $\mathcal{O}(p+q)$ steps by using the already obtained cost $C(P_{(v_{1,1},v_{1,2})},c,Q)$.
All  other costs can be computed iteratively in the same manner in $\mathcal{O}(p+q)$ steps. Since there are roughly $p \cdot q$ critical paths,
the calculation of $\widehat{pc}$ can be done in $\mathcal{O}(p^{2}q+pq^{2})$.
\end{proof}

The following result relates  linearizable instances on a directed acyclic graph having the same quadratic costs and source vertices, but  different target vertices.
\begin{lemma} \label{linear:gridlemma2}
Let $\mathcal{I} = (G,s,t,c,Q)$ be a  QSPP instance on the directed acyclic graph $G$.
We have that $c'$ is a linearization  vector of $\mathcal{I}$ if and only if the vector $c^{v}$ given by
\begin{equation}\label{linear:gridlemma2_eq}
c^{v}_{e} = \begin{cases}
c_{e}'  - 2 \cdot q_{(v,t),e} & \text{ if } e = (u,w) \text{ and } u \neq s \\
c_{e}' - 2 \cdot q_{(v,t),e} + c'_{(v,t)}  & \text{ if } e = (u,w) \text{ and } u = s
\end{cases},
\end{equation}
is a linearization vector of the instance $\mathcal{I}^{v} = (G,s,v,c,Q)$ for every vertex $v$ such that $(v,t) \in A$.
\end{lemma}
\begin{proof}
Let us assume that there is a vector $c'$ such that the vector $c^{v}$ defined in \eqref{linear:gridlemma2_eq} is a linearization vector
of $\mathcal{I}^{v}$ for \textcolor{black}{every vertex} $v$ such that $(v,t) \in A$.
Let $P = (e_{1},e_{2},\ldots,e_{k})$ be a $s$-$t$ path, where $e_{k} = (v,t)$.
Then $P^{v} = (e_{1},e_{2},\ldots,e_{k-1})$ is a $s$-$v$ path.
We have that
\begin{align*}
C(P,c') &=   c_{e_{1}}' + \sum_{i=2}^{k-1} c_{e_{i}}' + c_{e_{k}}' \\
& = c^{v}_{e_{1}} - c_{e_{k}}' +  2 \cdot q_{e_{k},e_{1}} + \sum_{i=2}^{k-1} (c^{v}_{e_{i}} + 2 \cdot q_{e_{k},e_{i}})  + c^{v}_{e_{k}}+ 2 \cdot q_{e_{k},e_{k}} \\
& = \sum_{i=1}^{k-1} c^{v}_{e_{i}} + 2 \cdot \sum_{i=1}^{k-1}  q_{e_{k},e_{i}}  + (-c_{e_{k}}'  + c^{v}_{e_{k}}+ 2 \cdot q_{e_{k},e_{k}}) \\
& = C(P^{v},c,Q)  + 2 \cdot \sum_{i=1}^{k-1} q_{e_{k},e_{i}}= C(P,c,Q).
\end{align*}
\textcolor{black}{Recall that  the linear cost vector $c$ is assumed to be zero.}
The fourth equation exploits  that $q_{e_{k},e_{k}} = 0, c^{v}_{e_{k}} = c_{e_{k}}'$.
This shows that $c'$ is a linearization vector of $\mathcal{I}$.

\textcolor{black}{The converse follows} in a similar way.
\end{proof}
Note that the above lemma is proven for any directed acyclic graph. Therefore, it is also valid for the directed grid graphs.
For the directed grid graphs,
we simplify notation and write $\mathcal{I}^{i,j}$ for the instance $\mathcal{I}^{v_{i,j}}$, and $c^{i,j}$
for the associated linear cost vector $c^{v_{i,j}}$. We also use $\mathcal{I}^{p,q} = \mathcal{I}$ and $c^{p,q} = c'$.
In what follows, we exploit the previous lemma to derive our linearization algorithm.

We first prove  that  any instance of the QSPP on  $G_{2,q}$  ($q\geq 2$) is linearizable.
\begin{lemma} \label{linear:LemmaG2q}
Let   $\mathcal{I} = (G_{2,q},s,t,c,Q)$ be a QSPP instance on a directed grid graph $G_{2,q}$ and $q \geq 2$.
Then  $\mathcal{I}$ is linearizable.
\end{lemma}
\begin{proof}
Let $P_{i}$ be the unique $s$-$t$ path containing arc $(v_{1,k},v_{2,k})$ for $k = 1,\ldots,q$.  Let us define the linear cost vector $c'$ as follows:
$$c_{e}' =  \begin{cases}
C(P_{1},c,Q) & \text{ if } e = (v_{1,1},v_{2,1}), \\
C(P_{k},c,Q) - C(P_{q},c,Q) & \text{ if } e = (v_{1,k},v_{2,k}) \text{ for some } k = 2,\ldots,q-1, \\
C(P_{q},c,Q) & \text{ if } e = (v_{1,1},v_{1,2}), \\
0 & \text{ otherwise}.
\end{cases}$$
One can readily see that $c'$ is a linearization of $\mathcal{I}$, and thus $\mathcal{I}$ is linearizable.
\end{proof}
Note that a QSPP instance $\mathcal{I} = (G_{2,q},s,t,c,Q)$ is linearizable also in the case that  $c$ is not a zero vector, see  Lemma \ref{linear:czero}.
Similarly one can prove that any QSPP instance on a directed grid graph $G_{p,2}$  is linearizable for every $p \geq 2$.

The following two results are the main  ingredients of our linearizability algorithm.
\begin{lemma} \label{linear:character1}
Let $\mathcal{I} = (G_{p,q},s,t,c,Q)$ be a QSPP instance.
Then $c'$ is a linearization vector of $\mathcal{I}$ if and only if
\begin{enumerate}[topsep=0pt,itemsep=-1ex,partopsep=1ex,parsep=2ex,label=(\roman*)]
\item $\widehat{c}^{p-1,q}$ is a linearization vector of the instance $\mathcal{I}^{p-1,q} = (G_{p,q},s,v_{p-1,q},c,Q)$,
\item $\widehat{c}^{p-1,j} = \widehat{pc}^{p-1,j}$ for $j = 1,\ldots, q-1$,
\end{enumerate}
where $\widehat{c}^{p-1,j}$ is the reduced form of the vector derived as in  \eqref{linear:gridlemma2_eq}, and $\widehat{pc}^{p-1,j}$ is the pseudo-linearization vector of $\mathcal{I}^{p-1,j}$.
\end{lemma}
\begin{proof}
Assume that $c'$ is a linearization vector of $\mathcal{I}$. Applying Lemma \ref{linear:gridlemma2} repeatedly to the instances $\mathcal{I}^{p,j}$ for $j = q,q-1,\ldots,1$, we get that $c'$ is a linearization vector of $\mathcal{I}^{p,q}$ if and only if $c^{p-1,j}$ derived as in Lemma \ref{linear:gridlemma2} is a linearization vector of $\mathcal{I}^{p-1,j}$ for $j = 1,\ldots,q$. Let $\widehat{c}^{p-1,j}$ be the reduced form of the linearization vector $c^{p-1,j}$. Since those two vectors are equivalent, we have that $c'$ is a linearization vector of $\mathcal{I}$ if and only if  $\widehat{c}^{p-1,j}$ is a linearization vector of $\mathcal{I}^{p-1,j}$ for $j = 1,\ldots, q$.

From Lemma \ref{linear:gridlemma2}, we also know that if $\mathcal{I}^{p-1,q}$ is linearizable, then $\mathcal{I}^{p-1,j}$ is linearizable for $j = 1,\ldots,q-1$, and in this case $\widehat{c}^{p-1,j}$ is a linearization vector of $\mathcal{I}^{p-1,j}$ if and only if $\widehat{c}^{p-1,j}$ equals to the pseudo-linearization vector $\widehat{pc}^{p-1,j}$ for $j = 1,\ldots,q-1$.
\end{proof}
Applying Lemma \ref{linear:character1} recursively to the instances $\mathcal{I}^{i,q}$ for $i = p,p-1,\ldots, 3$,
 we obtain the following schema for testing linearizability of a QSPP instance on the grid graph  $G_{p,q}$.
\begin{proposition} \label{linear:character2}
Let $\mathcal{I} = (G_{p,q},s,t,c,Q)$ be a QSPP instance on $G_{p,q}$. It holds that $c'$ is a linearization vector of $\mathcal{I}$ if and only if
\begin{enumerate}[topsep=0pt,itemsep=-1ex,partopsep=1ex,parsep=2ex,label=(\roman*)]
\item $\widehat{c}^{2,q}$ is a linearization vector of the instance $\mathcal{I}^{2,q} = (G,s,v_{2,q},c,Q)$,
\item $\widehat{c}^{i,j} = \widehat{pc}^{i,j}$ for $i= 2, \ldots , p-1$ and $j = 1,\ldots, q-1$,
\end{enumerate}
where $\widehat{c}^{p-1,j}$ is the reduced form of the vector derived as in  \eqref{linear:gridlemma2_eq}
by applying Lemma \ref{linear:character1} recursively. Here $\widehat{pc}^{p-1,j}$ is the pseudo-linearization vector of $\mathcal{I}^{p-1,j}$.
\end{proposition}

By exploiting Proposition \ref{linear:character2}, we derive an algorithm that verifies
if a QSPP instance on the directed grid graph $G_{p,q}$ is linearizable, see Algorithm \ref{alg:linear}.
Moreover our algorithm returns the linearization vector in reduced form, if such exists.

\begin{theorem}\label{linear:maintheorem}
The algorithm {\sc Linearize-grid-QSPP} determines if a QSPP instance  on the directed grid graph $G_{p,q}$  is linearizable,
and if so it  constructs its linearization vector  in  ${\mathcal{O}(p^{3}q^{2}+p^{2}q^{3})}$ time.
\end{theorem}
\begin{proof}
Recall that a QSPP instance is linearizable if and only if the pseudo-linearization vector is a linearization vector.
Therefore, the algorithm {\sc Linearize-grid-QSPP} iteratively applies Proposition \ref{linear:character2} to the pseudo-linearization vector in order to check linearizability
of the instance.

The algorithm involves computation of roughly $p\cdot q$ vectors $c^{ij}$, their reduced forms $\widehat{c}^{i,j}$, and the pseudo-linearization vectors $\widehat{pc}^{ij}$.	
It follows from Lemma \ref{linear:gridlemma2}, that  computing all the vectors $c^{ij}$ can be done iteratively.
From Lemma \ref{linear:reducedtime} we have that  the reduced form vectors $\widehat{c}^{i,j}$
obtained from the vectors $c^{ij}$ ($i=2, \ldots, p-1$, $j=1,\ldots, q-1$ ) can be computed in $\mathcal{O}(p^{2}q^{2})$.
From Lemma \ref{linear:pseudotime} it follows that the calculation of all the pseudo-linearization vectors $\widehat{pc}^{ij}$ requires $\mathcal{O}(p^{3}q^{2}+p^{2}q^{3})$.
The costs of all other calculations are small. Thus, the complexity of the algorithm is $\mathcal{O}(p^{3}q^{2}+p^{2}q^{3})$.
\end{proof}
To derive Algorithm 1, we assume that the linear cost vector is equal to the all-zero vector.
Clearly, our algorithm also works if the linear cost vector is not equal to zero, see Lemma \ref{linear:czero}.
The interested reader can download {\sc Linearize-grid-QSPP} and/or {\sc isLinearizable}
 algorithm from the following link and test linearizability of any QSPP instance on $G_{pq}$ ($p,q \geq 2$).
\begin{center}
{\url{https://huhao.org/}}
\end{center}
{\color{black} \cite{HaoSotirov17} generalize the approach from this section to all directed acyclic graphs.
In particular, they derive an algorithm that verifies whether a QSPP instance on a DAG is linearizable,
and present a new bounding scheme that exploits the corresponding linearization algorithm.}

\begin{algorithm}[H]
	\caption{\sc Linearize-grid-QSPP}\label{alg:linear}
	\begin{algorithmic}[H]
		\State \textbf{Input:} A QSPP instance $\mathcal{I} = (G_{p,q},v_{1,1},v_{p,q},c,Q)$.
		\State \textbf{Output:} The linearization vector of $\mathcal{I}$ if it exists.
		\Procedure{isLinearizable}{$\mathcal{I}$}
		\State $\widehat{pc} \leftarrow $ pseudo-linearization of $\mathcal{I}$
		\For{$i= p-1,\ldots,2$}
		\For{$j=1,\ldots,q-1$}
		\State $\widehat{pc}^{i,j} \leftarrow $ pseudo-linearization of $\mathcal{I}^{i,j}$ by using Lemma \ref{linear:pseudotime}
		\State {$\widehat{c}^{i,j}\leftarrow $ linear cost vector obtained as in Proposition \ref{linear:character2}} 
		\If {$\widehat{c}^{i,j} \neq \widehat{pc}^{i,j}$}
		\State \Return \texttt{false}
		\EndIf
		\State \textbf{end if}
		\EndFor
		\State \textbf{end for}
		\EndFor
		\State \textbf{end for}
		\State Calculate  $\widehat{c}^{2,q}$ using Prop.~\ref{linear:character2}, and pseudo-linearization $\widehat{pc}^{2,q}$ using Lemma \ref{linear:LemmaG2q}.
		\If {$\widehat{pc}^{2,q} \neq \widehat{c}^{2,q}$} 
		\State \Return \texttt{false}
		\EndIf
		\State \textbf{end if}
		\State \textbf{return} \texttt{true} and $\widehat{pc}$ 
		\EndProcedure
	\end{algorithmic}
\end{algorithm}

\section{Conclusion}
In this paper, we study the complexity and special cases of the quadratic shortest path problem.
In Theorem \ref{QAPtoQSPP2}, we present a polynomial-time reduction from the QAP to the QSPP.
 The size of the obtained QSPP instance is significantly smaller than the size of the instance obtained from the reduction provided in  \cite{frey2015quadratic}.
 Further, we give a new and simpler proof,  in comparison with the proof from \cite{rostami2016quadratic},
 showing that the general AQSPP cannot be approximated unless P=NP, see Theorem \ref{AQSPP-NPhard}.

Polynomial-time solvable  special cases of the QSPP are considered in Section \ref{section:linearQSPP}.
In Proposition  \ref{prop:weakSum},   we show that if the quadratic cost matrix is a symmetric  weak sum matrix and every $s$-$t$ path in $G$ has the same length, then the QSPP is linearizable.
In Proposition \ref{prop:productmatrix}, it is proven that if the quadratic cost matrix is a nonnegative symmetric product matrix, then the QSPP can be solved in $\mathcal{O}(m+n\log n)$ time.

 In Proposition \ref{necessary_complete}, we present necessary conditions for a QSPP instance on the complete digraph $K_{n}^{*}$ ($n\geq 4$) to be linearizable.
  These conditions turn out to be also sufficient for $K_{4}^{*}$, see Proposition \ref{linear:K4}.
  We also prove that every instance on a tournament with four vertices is linearizable, see Proposition \ref{linear:T4}.

   We provide a polynomial-time algorithm to check whether
  a QSPP instance on a directed grid graph is linearizable, see Theorem \ref{linear:maintheorem}.
The interested reader can download  our algorithm. \\

\noindent
{\bf Acknowledgements}  The authors would like to thank two anonymous referees for suggestions that led to an improvement of this paper.


\end{document}